\documentclass[12pt,a4j]{article}
\input{1112Univ_sch.STY}
\begin{document}
\title{Universality of the category of schemes}
\author{Satoshi Takagi}
\date{}
\maketitle

\begin{abstract}
In this paper, we generalize the construction
method of schemes to other algebraic categories,
and show that the category of coherent schemes
can be characterized by a universal property,
if we fix the class of Grothendieck topology.
Also, we introduce the notion of $\scr{C}$-schemes,
which is a further generalization of coherent schemes
and still shares common properties with ordinary
schemes.
\end{abstract}

\tableofcontents

\section{Introduction}

In this paper, we generalize the construction of schemes
to other algebraic categories.

There are already many attempts of extending the category
of schemes, for various reasons.
One of it is from the geometry over $\FF_{1}$:
the rough idea is to construct a scheme-like topological object
from monoids.
There are numerous attempts towards this goal,
and we cannot list them all.
Here, we just mention that
there is a survey on this topic \cite{Lorscheid}.
In this article, we focus on the relation between
the construction of ours and those of
To\"{e}n-Vaqui\'{e} \cite{TV}, and Deitmar \cite{Deitmar1}.
(the latter two constructions are actually equivalent,
according to \cite{Lorscheid}.)

However, not many have mentioned its universal property
of the proposed new objects and categories.
To claim that the new objects are ``good",
we must make it explicit how it is similar to schemes.
In other words, we should ask to ourselves,
\textit{
\begin{center}
``What is a scheme, anyway?"
\end{center}
}
We know how schemes are useful, but we don't know why.
Despite its heavy and complicated machinary,
the category of schemes behaves surprisingly good.
This implies that the construction of schemes,
or almost equivalently, the spectrum functor,
should be realized as an adjoint of a simpler functor.

In this paper, we go back to this fundamental question:
we give the universal characterization of the category of
coherent schemes.

The universality cannot be obtained, only by specifying
the algebraic object: even if we restrict ourselves to rings,
there are several ways of defining a scheme-like objects.
Therefore, we must designate its Grothendieck topology.

Also, we must be aware that the category
of schemes is not complete, although the category of rings
is co-complete.
Therefore, in order to give a universal property,
we should be able to apply the spectrum functor to
schemes:
if we call the scheme-like $V$-valued spaces
as $V$-schemes, then we must be able to define
$\cat{$V$-Scheme}^{\op}$-schemes.

This implies that the usual definition of schemes is not appropriate:
it uses the notion of local rings and local homomorphisms
where infinite operations occur, which is not available
for schemes in general.
Therefore, we give another way of describing this local property,
only using finite operations on the value category.
This idea is already mentioned in the previous preprint
of the author \cite{Takagi1}.
The advantage of this definition is that we can define
the spectrum functor as an adjoint of the global section functor,
which is by nature, much simpler.
Another advantage is that this gives a way to
define a larger category containing all coherent schemes, which is
complete: see \cite{Takagi2}.

Here, we will list up the main results.
Let $\scr{C}$ be a category with pull backs and
finite coproducts, equipped with a Grothendieck
topology $(\scr{E},\mathcal{O})$
(the precise definition is given in \S 3).
We call the triple $(\scr{C},\scr{E},\mathcal{O})$
a coherent site.
This is \textit{schematic},
if it admits finite open patchings
(again, the definition is given in \S 3).

\begin{Thm}
Let $\cat{CohSite}$, $\cat{CohSch}$ be the (2-)categories
of coherent sites, schematic
coherent sites, respectively.
\begin{enumerate}
\item The underlying functor 
$U:\cat{CohSch} \to \cat{CohSite}$
admits a left adjoint $\func{Sch}$.
The unit is the spectrum functor 
$\Spec:\scr{C} \to \cat{$\scr{C}$-Sch}=\func{Sch}(\scr{C})$.
(Theorem \ref{thm:univ:coh:sch})
\item When $\scr{C}$ is the opposite category
of commutative rings, with the topology
induced from ideals, then $\scr{C}$-schemes
are coherent schemes.
(Theorem \ref{thm:sch:local:space})
\end{enumerate}
\end{Thm}

The universal property of the Zariski topology
is summarized as follows:
\begin{Thm}[Theorem \ref{thm:ideal:top:mod}]
Let $V$ be a self enhancing algebraic type
with a constant operator,
and $W$ be the type of commutative monoid objects
in the category of $V$-algebras.
If the morphism class $\scr{E}$ consists
of localizations of finite type, then
the Zariski topology is the coarsest topology
satisfying the following condition:

For any $W$-algebra $R$ and $R$-module $M$,
the induced $\scr{O}_{\Spec R}$-module $\Shf^{\Zar}(M)$
is zero if and only if $M=0$.
\end{Thm}

Note that this paper is only a stepping stone
to various generalizations.
However, we decided to rewrite it from the first step:
since although the scheme theory is already a classical topic,
its universal property is scarcely discussed,
and we need it to be clarified to
apply its machinary to other workfields.

This paper is organized as follows:
in section \S 2, we give a brief summary of
algebraic types and lattice theories.
The reader may skip this section, if he or she
is familiar with the terminology.

In \S 3, we discuss the universal property of
$\scr{C}$-schemes: namely it is given by the
universal property of schematic topological category
under $\scr{C}$.
In the process, we introduce the notion
of weak $\scr{C}$-schemes, which is the generalization
of $\scr{A}$-schemes.
Using this, we obtain the spectrum functor
as an adjoint of the global section functor.

In \S 4, we compare our notion of $\scr{C}$-schemes
with the conventional coherent schemes.
This is somewhat time-consuming, since
the definitions of the two are very different.
Here, we also show the correspondence between
the ideals and the Grothendieck topology.

Finally in \S 5, we discuss the universal
property of ideal topology.
Note that the ground of choosing the
ideal topology is only given by looking
at modules.
When we talk of algebraic objects other than
rings, we know that congruences take place of ideals.
However, the result obtained here shows that
this is a groundless fear, since the
ideal topology detects modules similarly
as in the case of rings.

\textbf{Notation and Conventions:}
In the sequel, we assume the reader
to be familiar with category-theoretic
languages; the textbook \cite{CWM} will be sufficient
for this purpose.
We fix a universe, and do not mention it
when unnecessary.
Any ring is unital,
and $\NNN$ is the set of non-negative rational integers.

\textbf{Acknowledgements:}
We would like to thank H. Minamoto for
valuable discussions.
I was partially supported by the Grant-in Aid for
Young Scientists (B) \# 23740017.

\section{Preliminaries}

\subsection{Algebras}

Here, we begin with the definition
of algebraic types.
Conventionally, an algebraic type
is a pair $\langle \Omega,E \rangle$
of the set of operators and
equational classes.
However, to define equational classes
(and therefore, derived operators) explicitly,
it requires some terminology of
\textit{trees}, or \textit{languages}.
To emphasize the relation between
algebras and category theories,
we decided to describe these in categorical terms,
as below. 

\begin{Def}
\begin{enumerate}
\item
Let $\cat{Cat$^{\times}$}$ be the category
of small categories with finite products
(here, we include null products,
namely we also assume the existence of 
the terminal object)
and product-preserving functors.
We have a left adjoint $\func{free}$ of the underlying functor
$\obj:\cat{Cat$^{\times}$} \to \cat{Set}$
which sends a small category $\scr{C}$
to its object set $\obj(\scr{C})$.
\item
We set $\scr{X}=\func{free}(*)$.
This is actually equivalent to
the opposite category of finite sets and maps.
We denote by $[n] \in \obj(\scr{X})$
the set $\{1,2,\cdots, n\}$.
\item 
An object in the overcategory $\cat{Set $\downarrow\NNN $}$
is called a \textit{graded set}.

For an object $a:\Omega \to \NNN$
in $\cat{Set$\downarrow\NNN $}$, 
we denote by $\Omega_{n}$ for the
inverse image $a^{-1}(n)$ for a non-negative integer $n$.
\end{enumerate}
\end{Def}

\begin{Prop}
Let $U:\cat{$\scr{X} \downarrow $Cat$^{\times}$}
\to \cat{Set $\downarrow \NNN $}$
be a functor defined by
\[
[\scr{X} \stackrel{j}{\to} \scr{C}]
\mapsto [\amalg_{n \in \NNN}\scr{C}(j[n],j[1]) \to \NNN],
\]
where the latter map is determined
by sending elements of $\scr{C}(j[n],j[1])$ to $n$.
A morphism $\scr{C} \to \scr{D}$ in 
$\cat{$\scr{X} \downarrow $Cat$^{\times}$}$
canonically gives a map
$\amalg_{n \in \NNN}\scr{C}(j[n],j[1])
\to \amalg_{n \in \NNN}\scr{D}(j[n],j[1])$ over $\NNN$.
Later, we omit $j$ if there seems to be no confusion.
Then, $U$ admits a left adjoint.
\end{Prop}
\begin{proof}
It is straightforward to see that $\cat{Cat$^{\times}$}$
is small complete.
Therefore, it remains to show that
the solution set condition is satisfied 
to apply Freyd's adjoint functor theorem.
Suppose $\Omega \to \NNN$ is an object of 
$\cat{Set$\downarrow \NNN $}$
and $\scr{X} \to \scr{C}$ an object
of $\cat{$\scr{X} \downarrow $Cat$^{\times}$}$
with a map $f:\Omega \to \amalg_{n}\scr{C}([n],[1])$
over $\NNN$.
Then there exists the smallest subcategory
$\scr{C}_{0}$ of $\scr{C}$ which is
an object of $\cat{$\scr{X} \downarrow $Cat$^{\times}$}$,
and $f$ factors through $\amalg_{n}\scr{C}_{0}([n],[1])$.
It is clear that the set of these $\scr{C}_{0}$'s
$\{\scr{C}_{0}\}_{\scr{C} \in 
\cat{$\scr{X} \downarrow $Cat$^{\times}$}}$
modulo equivalence is small.
Hence, the solution set condition is satisfied
and we have the left adjoint of $U$.
\end{proof}

\begin{Def}
\begin{enumerate}
\item We denote again by $\func{free}$
the left adjoint functor
$\cat{Set $\downarrow \NNN $} \to
\cat{$\scr{X} \downarrow $Cat$^{\times}$}$ 
obtained in the above proposition.
For a graded set $\Omega$,
$\func{free}(\Omega)$
is the \textit{free algebraic type}
generated by $\Omega$.
The graded set
\[
D(\Omega)=\amalg_{n \in \NNN}
\func{free}(\Omega)([n],[1]) \to \NNN
\]
is the set of \textit{derived operators}
generated by $\Omega$.

\item
A \textit{\rom{(}finitary\rom{)} algebraic type} is 
a pair $\langle\Omega,E\rangle$, where
\begin{enumerate}
\item $\Omega$ is a graded set, and
\item $E=(E_{n})_{n}$ is a sequence
of sets indexed by $\NNN$,
and $E_{n}$ is an equivalence relation on $D(\Omega)_{n}$.
\end{enumerate}
$\Omega$ is called the set of \textit{operators}
of $V$ and the elements in $\Omega_{n}$
are called \textit{$n$-ary operators}.
$E$ is called the \textit{equational class} of $V$.

\item
For a finitary algebraic type $V=\langle\Omega,E\rangle$,
let $\cat{$V\downarrow $Cat$^{\times}$}$
be the full subcategory of 
$\cat{$\func{free}(\Omega) \downarrow $Cat$^{\times}$}$,
consisting of objects $f:\func{free}(\Omega) \to \scr{C}$
such that $f(\varphi)=f(\psi)$
for any $(\varphi,\psi) \in E_{n}$.
The existence of the initial object in 
$\cat{$V\downarrow $Cat$^{\times}$}$
can be shown by applying Freyd's adjoint theorem again,
and is denoted by $\scr{X}_{V}$.

\item
A morphism $V \to W$ of algebraic types is
a morphism $\scr{X}_{V} \to \scr{X}_{W}$
in $\cat{$\scr{X}\downarrow$Cat$^{\times}$}$.


\item For any algebraic type $V$ and a
symmetric monoidal category $(\scr{C},\otimes)$,
a \textit{$V$-algebra} in $\scr{C}$ is a morphism
$A:V \to \scr{C}$ of monoidal categories.
We often just indicate $A([1]) \in \scr{C}$ for $A$,
if there seems to be no ambiguity.

A \textit{morphism of $V$-algebras} in $\scr{C}$
is a $\otimes$-preserving natural transformation.
We denote by $\cat{$V$-alg($\scr{C}$)}$
the category of $V$-algebras in $\scr{C}$.
\item When $\scr{C}$ is the category $\cat{Set}$ of small sets, 
regarded as a cartesian closed category,
then we omit the indication of $\scr{C}$.
\end{enumerate}
\end{Def}

\begin{Rmk}
\begin{enumerate}
\item To present a fixed type of algebra
as a functor is already realized by Lawvere,
known as the \textit{Lawvere theory}.
Moreover, it is true that this definition
of algebras is quite complicated,
and category theorists would rather treat
finitary algebraic category instead;
see Definition \ref{def:finitary:alg} below.
It is also true that the following construction
of scheme theories is available 
in this generalization \cite{Durov}.
However, here we reviewed the classical definition
of algebras (Lawvere theory and finitary algebraic categories
are generalizations of algebras, but are not equivalent),
since anyway we need to clarify that
these generalizations are indeed a generalization
of algebras.
\item The reader should be aware that,
when the monoidal structure $\otimes$
of a symmetric monoidal category $\scr{C}$
is not Cartesian, the notion of $V$-algebras in $\scr{C}$
introduced above may be different from the
conventional one.
For example, if $\scr{C}$ is the category
of abelian groups with the tensor structure
and $V$ is the algebraic type of monoids,
then $\cat{$V$-alg($\scr{C}$)}$ does \textit{not}
coincide with the category of unital associative rings,
since a $V$-algebra $A$ in $\scr{C}$ 
in the above sense is also required
to have first and second projections
$A \otimes A \to A$, the diagonal mapping $A \to A \otimes A$,
and so on.

This also shows that some of the algebraic type
cannot be realized, using the modern fashion of 
operads\cite{Markl}:
for example, the algebraic type of groups
cannot be described as an operad,
since the inverse law $aa^{-1}=1$ requires
multiple times of references to one variable $a$.
\end{enumerate}
\end{Rmk}

\begin{Def}
\label{def:finitary:alg}
\begin{enumerate}
\item
Let $F: \scr{X} \rightleftarrows \scr{A}:U$
be an adjunction between two categories,
and set $T=UF:\scr{X} \to \scr{X}$.
Let $\scr{X}^{T}$ be the category of $T$-algebras,
namely objects $X$ in $\scr{X}$ with
a $T$-action $TX \to X$ and $T$-equivariant morphisms.
We have a natural functor $K:\scr{A} \to \scr{X}^{T}$,
called the \textit{comparison functor}.
The above adjoint is \textit{monadic},
if $K$ is an equivalence of categories.
\item A \textit{finitary algebraic category}
is a category $\scr{A}$ equipped with a adjunction
$\cat{Set} \rightleftarrows \scr{A}$ which is monadic.
\end{enumerate}
\end{Def}

Let $f:V \to W$ be a morphism of algebraic types.
Then, $f$ induces an underlying functor
$U:\cat{$W$-alg} \to \cat{$V$-alg}$ which
has a left adjoint $\func{free}$.
In particular,
we have an adjoint $\func{free}:\cat{Set} \rightleftarrows \cat{$V$-alg}$
for any algebraic type $V$, and this is monadic.
Consequently, $\cat{$V$-alg}$
is small complete and small co-complete.
We we refer to the funtor $\func{free}$ as
the \textit{free generator}.

\subsection{Lattice theories}
In this section, we give a brief summary
of lattice theories.
This also emphasizes the fact that restricting to coherent schemes 
is somewhat better than considering all schemes.

\begin{Def}
A topological space $X$ is \textit{sober},
if any irreducible closed subset has a unique generic point.
$X$ is \textit{quasi-separated}
(\cite{EGA4}, Proposition 1.2.7), if the intersection of any
two quasi-compact open subset of $X$ is again quasi-compact.
$X$ is \textit{coherent}, if it is sober,
quasi-compact, quasi-separated, and
admits a quasi-compact open basis.
\end{Def}

\begin{Def}
A poset $(L, \leq)$ is a \textit{distributive lattice}, if 
\begin{enumerate}[(a)]
\item any two elements of $a,b$ of $L$ admit
a \textit{join} $a \vee b$, namely the supremum of $\{a,b\}$,
and a \textit{meet} $a \wedge b$,
namely the infimum of $\{a,b\}$,
\item $L$ has a unique maximal (resp. minimal) element
$1$ (resp. $0$), and
\item distribution law holds:
$x \wedge (y \vee z)=(x \wedge y)\vee(x \wedge z)$.
\end{enumerate}
A distributive lattice $L$ is \textit{complete},
if the supremum is defined for any subset of $L$,
and infinite distribution law holds.
\end{Def}
In fact, a distributive lattice can be regarded as a semiring:
recall that, an element $a$ of a monoid $M$
is \textit{absorbing}, if $ax=a$ for any $x \in M$.
\begin{Def}
A 5-uple $R=(R,+,\times, 0,1)$ is an \textit{\rom{(}idempotent\rom{)} semiring} if
\begin{enumerate}[(a)]
\item $R$ is a set, $+,\times$ are two binary operators
on $R$, and $0,1$ are two elements of $R$,
\item $(R, +, 0)$ is a commutative idempotent monoid,
and $(R,\times,1)$ is a commutative monoid,
\item $0$ (resp. $1$)
is an absorbing element with respect to the multiplication
(resp. the addition), and
\item distribution law holds.
\end{enumerate}
\end{Def}
Let $R$ be a semiring.
Then, $R$ can be regarded as a poset by setting
\[
a \leq b \Leftrightarrow a+b=b.
\]
Then $a+b$ gives the supremum of $a,b$ with respect
to this order, and $0,1$ become the maximum
and the minimum of $R$, respectively.
In fact,
\begin{Prop}
To give a distributive lattice is equivalent
to giving a semiring with idempotent multiplication.
\end{Prop}
Indeed, if we are given a semiring $R$
with idempotent multiplication,
then we may replace $+$ by $\vee$ and
$\times$ by $\wedge$.

In particular, a distributive lattice can be regarded
as an algebra.
Let $\cat{DLat}$ (resp. $\cat{CDLat}$) be the category
of distributive lattices (resp. complete distributive lattices)
and their homomorphisms.
These categories are finitary algebraic categories.
As a consequence, $\cat{DLat}$ is small complete
and small co-complete.

The underlying functor $U:\cat{CDLat} \to \cat{DLat}$
admits a left adjoint $\func{comp}$, defined as follows:
for a distributive lattice $L$, $\func{comp}(L)$
is the set of all non-empty \textit{lower sets} of $L$:
\[
\func{comp}(L)=\{ \emptyset \neq S \subset L \mid 
x \leq y,\ y\in S \Rightarrow x \in S\}.
\] 
$S \vee T$ (resp. $S \wedge T$) is defined
as the lower set generated by $s \vee t$ (resp. $s \wedge t$)
for $s \in S$, $t \in T$ respectively.

Let $\mathbf{1}$ be the initial object in $\cat{DLat}$.
This is the simplest Boolean lattice $\{0,1\}$.
\begin{Def}
\begin{enumerate}
\item
For a topological space $X$,
$\Omega(X)$ is the set of open subsets of $X$.
This becomes a complete distributive lattice
via setting $\vee=\cup$ and $\wedge=\cap$,
and a continuous map $f:X \to Y$ between
two topological spaces induces a lattice homomorphism
$f^{-1}:\Omega(Y) \to \Omega(X)$.
Therefore, we have a contravariant functor
$\Omega:\cat{Top}^{\op} \to \cat{CDLat}$.
\item
Conversely, for a complete distributive lattice
$L$, let $\func{pt}(L)$ be the set of
homomorphisms $L \to \mathbf{1}$.
This has a natural topology, the open set of which
is of the form
\[
\phi(a)=\{ p \in \func{pt}(L) \mid p(a)=1\}
\]
where $a \in L$.
Then $\func{pt}(L)$ becomes a sober space.
If we denote the category of sober spaces by $\cat{Sob}$,
then we have a contravariant functor 
$\func{pt}:\cat{CDLat} \to \cat{Sob}^{\op}$.
\end{enumerate}
\end{Def}
The functor $\Omega$ is not essentially surjective.
The objects in the image category of $\Omega$
is called \textit{spatial}, and we denote by
$\cat{SCDLat}$ the full subcategory of $\cat{CDLat}$
consisting of spatial complete distributive lattices.
It turns out that $\Omega$ and $\func{pt}$
gives an equivalence of categories 
$\cat{Sob}^{\op} \simeq \cat{SCDLat}$.
This also gives the right adjoint of the
underlying functor $\cat{Sob} \to \cat{Top}$,
namely the \textit{soberification} $\func{pt}\Omega$,
and the left adjoint of $\cat{SCDLat} \to \cat{CDLat}$.
Summarizing, we have a commutative diagram of adjunctions:
\[
\xymatrix{
\cat{Top}^{\op} \ar@<.5ex>[r] \ar[d]_{\Omega} & 
\cat{Sob}^{\op} \ar@<.5ex>[l]^{U} \ar@{<->}[d]_{\Omega}^{\simeq} \\
\cat{CDLat} \ar@<.5ex>[r] &
\cat{SCDLat} \ar@<.5ex>[l]^{U}
}
\]
Let us see for distributive lattices.
The correspondence between the spaces
and lattices becomes more clear in this case.

For a distributive lattice $L$, $\func{comp}(L)$
turns out to \textit{have enough points},
namely, $\func{comp}(L)$ is spatial.
This is a consequence of a more general statement:
\begin{Prop}
Let $R$ be a semiring,
and $I$ an ideal of $R$ which is not the unit ideal.
Then, there exists a maximal ideal containing $I$.
\end{Prop}
Actually, we use this proposition when
we prove the existence of a maximal ideal of a given ring.

When $R$ is a distributive lattice,
the ideals of $R$ correspond to the lower sets of $R$.

\begin{Thm}[Stone duality]
\begin{enumerate}
\item
For a distributive lattice $L$,
$\func{pt}\func{comp}(L)$ becomes a coherent space,
and a homomorphism $L \to M$
induces a quasi-compact morphism 
$\func{pt}\func{comp}(M) \to \func{pt}\func{comp}(L)$.
Hence, we have a functor $\func{pt}\func{comp}:\cat{DLat}
\to \cat{Coh}^{\op}$.
\item
Conversely, for a coherent space $X$,
let $\Omega_{c}(X)$ be the set of quasi-compact open subsets
of $X$.
Then, $\Omega_{c}(X)$ becomes a distributive lattice and
$\Omega_{c}$ induces a functor
$\cat{Coh}^{\op} \to \cat{DLat}$.
This gives the inverse of $\func{pt}\func{comp}$,
and hence an equivalence of categories.
\end{enumerate}
\end{Thm}
The functor $\func{comp}$ corresponds to the
underlying functor $\cat{Coh} \to \cat{Sob}$:
note that this is \textit{not} fully faithful,
as we only consider quasi-compact morphisms between
coherent spaces.

Summarizing, we have a commutative square:
\[
\xymatrix{
\cat{Coh}^{\op} \ar@<.5ex>[r]^{U} 
\ar@{<->}[d]_{\simeq}^{\Omega_{c}} &
\cat{Sob}^{\op} \ar@<.5ex>[l] \ar@{<->}[d]_{\simeq}^{\Omega} \\
\cat{DLat} \ar@<.5ex>[r] & \cat{SCDLat} \ar@<.5ex>[l]^{U}
}
\] 
Since $\cat{DLat}$ is small complete and co-complete,
so is $\cat{Coh}$.
This justifies our standing point that we stick on to
coherent spaces.

\begin{Rmk}
Here, we will explain some other aspects,
which give reasonings of sticking to coherent spaces.
\begin{enumerate}
\item
Let $R=k^{\prod X}$ be a direct product
of the copies of a field $k$, with an infinite index set $X$.
Note that $\Spec R \neq \amalg_{X} \Spec k=X$,
since $\Spec R$ is quasi-compact, while $X=\amalg_{X} \Spec k$
is not: $X$ is discrete,
and $\Spec R=\beta X$, the Stone-\v{C}ech compactification
of $X$. $\beta X$ is a Hausdorff coherent space,
in other words, a totally disconnected
compact Hausdorff space.
A point $x$ on $\beta X \setminus X$ is a prime ideal corresponding
to a non-principal ultrafilter on $X$.
The stalk $\scr{O}_{X,x}$ is a 
non-standard extension field of $k$,
if the cardinality of $k$ is not less than $\# X$.
(See for example, \cite{CN} or other textbooks on set theory
or model theory.)

This is saying that considering coproducts of the
underlying space in $\cat{Top}$ might not 
be the best choice, since product of rings and
coproduct of schemes does not coincide.
In contrast, when we consider coproducts in $\cat{Coh}$,
then $\amalg_{X} \Spec k$ coincides with $\Spec R$.
This implies that we are ignoring some data of the algebra,
when we regard the underlying space merely as a topological space,
especially when we consider infinite operations.

\item One of the motivation of extending the
notion of schemes is from arithmetics:
we want to handle the infinite places of a given number field
equally with the finite places.
However, the archimedean-complete field $\RR$
appears only after some transcendental operations.
Let $R$ be a subring of $\QQ^{\NNN}$
whose element $(a_{n})_{n \in \NNN}$ is uniformly bounded:
namely, there is an upper bound $M>0$ such that
$|a_{n}|<M$ for any $n$.
Then, $\RR$ appears as the residue field 
$R/\mathfrak{m}$ of
 the maximal ideal $\mathfrak{m}$
of $R$, corresponding to a non-principal ultrafilter $\scr{U}$ on $\NNN$:
\[
\mathfrak{m}=\{(a_{n})_{n} \mid \forall\epsilon>0,
\ |a_{n}|<\epsilon \ \rom{a.e.} \scr{U}\}.
\]
This hints us that, these transcendental points
may bridge the gap between the archimedean world
and the non-archimedean world.
\end{enumerate}
\end{Rmk}

Before we go on to the next section,
we mention a simple, but important
fact on sheaves on coherent spaces.
Let $X$ be a topological space.
Since $\Omega(X)$ and $\Omega_{c}(X)$
are posets, they can be regarded as a category:
the object set is $\Omega(X)$ (resp. $\Omega_{c}(X)$) and
there is a unique morphism
$a \to b$ if and only if $a \leq b$.
\begin{Prop}
Let $X$ be a coherent space,
and $\scr{C}$ be a small complete category.
Then, to give a $\scr{C}^{\op}$-valued sheaf on $X$
is equivalent to give a finite continuous functor
$\Omega_{c}(X) \to \scr{C}$. 
\end{Prop}
\begin{proof}
Let $F: \Omega_{c}(X) \to \scr{C}$ be a finite continuous
functor. We will define a continuous functor
$\tilde{F}:\Omega(X) \to \scr{C}$ as follows:
any $U \in \Omega(X)$ is a union of quasi-compact
subsets $U_{i}$'s of $X$.
Then, $\tilde{F}(U)$ is defined as the equalizer of
$\prod_{i} F(U_{i}) \rightrightarrows \prod_{ij} F(U_{i} \cap U_{j})$.
This does not depend on the choice of the refinement $\{U_{i}\}_{i}$.
\end{proof}

\section{Universality of coherent schemes}
In this section, we characterize the category
of coherent schemes by a universal property.
Note that this cannot be achieved only
by specifying the algebraic type;
we must also fix the Grothendieck topology.

The essential idea of the construction
of general schemes is already given in 
\cite{TV}.

\subsection{The definition of $\scr{C}$-schemes}

In the sequel, we fix the following data:
$\scr{C}$ is a category
which admits pull backs and finite coproducts.
\begin{Def}
Let $\scr{E}$ be a lluf subcategory of $\scr{C}$
(namely, $\scr{E} \to \scr{C}$ is faithful
and essentially surjective) such that
\begin{enumerate}[(a)]
\item
$\scr{E}$ is closed under isomorphisms,
and base change: namely,
if $a \to b$ is a morphism in $\scr{E}$,
then $a \times_{b} c \to c$ is also in $\scr{E}$.
\item $\scr{E}$ admits pushouts, and $\scr{E} \to \scr{C}$
is pushout preserving.
In particular, the initial object of $\scr{C}$
coincides with that of $\scr{E}$.
\item Any morphism $f:A\to R$ in $\scr{E}$
is monic and \textit{flat}, namely $A \times_{R}(-)$
is coproduct preserving.
\item If $\{U_{i} \to X\}_{i}$ is a finite set of $\scr{E}$-morphisms,
then the diagram
\[
\amalg_{i,j}U_{i} \times_{X} U_{j}
\stackrel{p_{1},p_{2}}{\rightrightarrows} \amalg_{i}U_{i}
\]
has the coequalizer in $\scr{E}$,
where $p_{1}:U_{i} \times_{X} U_{j} \to U_{i}$
and $p_{2}:U_{i} \times_{X} U_{j} \to U_{j}$
are canonical morphisms.
\end{enumerate}
\end{Def}

\begin{Def}
\label{def:groth:top}
$\mathcal{O}=\{\mathcal{O}_{X}\}_{X \in \scr{C}}$
is a family indexed by the objects in $\scr{C}$,
and $\mathcal{O}_{X}$ is a family of finite sets
$\{U_{i} \to X\}_{i}^{<\infty}$ of morphisms in $\scr{E}$,
which satisfies the following condition:
\begin{enumerate}
\item
the canonical fork
\[
\amalg_{i,j}U_{i} \times_{X} U_{j}
\stackrel{p_{1},p_{2}}{\rightrightarrows} \amalg_{i}U_{i} \to X
\]
is exact.
This is what we call \textit{descent datum}.
\item $\mathcal{O}$ gives a coherent Grothendieck topology
on $\scr{C}$, namely:
\begin{enumerate}
\item $\{X \to X\}$ is in $\scr{O}_{X}$.
\item If $S_{1} \subset S_{2}$ is an inclusion
of finite sets of $\scr{E}$-morphisms over $X \in \scr{C}$
and $S_{1} \in \mathcal{O}_{X}$,
then $S_{2} \in \mathcal{O}_{X}$.
\item If $X \to Y$ is a morphism in $\scr{C}$,
and $S \in \mathcal{O}_{Y}$,
then $X \times_{Y} S \in \mathcal{O}_{X}$,
where
\[
X \times_{Y} S=\{ X \times_{Y} U \to X \mid 
[U \to Y] \in S\}.
\]
\item If $S \in \mathcal{O}_{X}$
and $T=\{U_{i} \to X\}_{i}$ are finite sets of $\scr{E}$-morphisms
such that $V \times_{X} T \in \mathcal{O}_{V}$
for any $[V \to X] \in S$, then
$T \in \mathcal{O}_{X}$.
\end{enumerate}
\end{enumerate}
\end{Def}

\begin{Def}
\begin{enumerate}
\item
We will refer to the triple $\scr{C}=(\scr{C},\scr{E},\mathcal{O})$
as a \textit{coherent site}.
\item
Let $\scr{C}_{i}=(\scr{C}_{i},\scr{E}_{i},\mathcal{O}^{i})$
($i=1,2$) be two coherent sites.
A morphism of coherent sites
$F:\scr{C}_{1} \to \scr{C}_{2}$ is a functor
such that $F(\scr{E}_{1}) \subset \scr{E}_{2}$
and $F(\mathcal{O}^{1}) \subset \mathcal{O}^{2}$.
\item We denote by $\cat{CohSite}$
the 2-category of coherent sites. 
\end{enumerate}
\end{Def}

\begin{Exam}
Here, we will give some examples of coherent sites.
\begin{enumerate}
\item The category $\cat{Coh}$ of coherent spaces,
or equivalently, the opposite category $\cat{DLat}^{\op}$
of distributive lattices: $\scr{E}$ is the subcategory
of open immersions, and $\{U_{i} \to X\} \in \mathcal{O}$
if and only if $\{U_{i}\}$ covers $X$, in the usual sense.
\item The opposite category $\cat{CRing}^{\op}$ of commutative
ring with units:
$\scr{E}$ is the subcategory of localizations
$S^{-1}R \to R$ of finite type, namely
the multiplicative system $S$ is finitely
generated (hence $S^{-1}R \simeq R_{f}$ for some $f \in R$)
$\{R_{f_{i}} \to R\}_{i} \in \mathcal{O}_{R}$
if and only if $(f_{i})_{i}$ generates the unit ideal.
We will discuss the generalization of this example
in section \S 5.
\item The opposite category $\cat{CMnd$_{0}$}^{\op}$
of commutative monoids with an absorbing element.
$\scr{E}$ is the subcategory of localizations of finite type,
similar to the case of rings.
$\{R_{f_{i}} \to R\}_{i}$ is an element of $\mathcal{O}_{R}$ if and only if
one of the $f_{i}$'s is a unit.
\end{enumerate}
\end{Exam}

The crucial difference between the category
of rings and that of schemes is that we can ``patch"
objects along open subobjects.
Here, we will axiomatize what ``patching" means.
\begin{Def}
Let $\scr{C}=(\scr{C},\scr{E},\mathcal{O})$
be a coherent site.
$\scr{C}$ is \textit{schematic},
if any cocartesian diagram
\[
\xymatrix{
a \ar[r] \ar[d] & b \ar[d] \\
c \ar[r] & b \amalg_{a} c
}
\]
in $\scr{E}$ is bicartesian, namely $a$ is the pullback of 
$[b \to b\amalg_{a} c \leftarrow c]$.
\end{Def}
We will denote by $\cat{CohSch}$
the full subcategory of $\cat{CohSite}$
consisting of schematic coherent sites.
The main theorem of this paper is as follows:

\begin{Thm}
\label{thm:univ:coh:sch}
The underlying functor 
$U:\cat{CohSch} \to \cat{CohSite}$
admits a left adjoint $\func{Sch}$.
\end{Thm}
The next section is devoted to the proof.

\begin{Rmk}
\label{rmk:sch:cat:incomp}
Note that until now, we haven't assumed completeness
of the category $\scr{C}$.
This enables us to apply $\func{Sch}$ to the
category of schemes, and we can say that
the above adjoint is idempotent.
Note that the category of ordinary schemes is \textit{not} complete,
for example, we cannot define an infinite product of
projective spaces $\prod^{\infty}\PP^{1}$.
The reason of this incompleteness relies on the fact
that we assume schemes to be locally isomorphic
to the spectrum of a ring:
$\prod^{\infty}\PP^{1}$ has the product topology,
which means that there are no open affine subsets.
This is another motivation of considering weak $\scr{C}$-schemes
later on.
\end{Rmk}

\subsection{Proof of Theorem \ref{thm:univ:coh:sch}}

To begin with, we must construct
the spectrum $\Spec^{0} R$ for each object $R$ 
of a coherent site $\scr{C}$. 
\begin{Def}
Let $\scr{C}$ be a coherent site,
and $R \in \scr{C}$.
\begin{enumerate}
\item
$\Omega_{0}(R)$ is the family of finite sets 
$\{U_{i} \to R\}_{i}$ of $\scr{E}$-morphisms.
\item We define a relation $\prec$ on $\Omega_{0}(R)$ by
\[
\{U_{i}\}_{i} \prec \{V_{j}\}_{j}
\Leftrightarrow \{U_{i}\times_{R} V_{j} \to U_{i}\}_{j} \in 
\mathcal{O}_{U_{i}} \ (\forall i).
\]
This relation satisfies the reflexivity and transitivity.
\item Let $\equiv$ be the equivalence relation
generated by $\prec$, namely
\[
a \equiv b \Leftrightarrow a \prec b \rom{ and } b \prec a.
\]
Set $\Omega_{1}(R)=\Omega_{0}(R)/\equiv$.
\item We can define $\vee$ and $\wedge$ on $\Omega_{1}(R)$ by
\begin{eqnarray*}
\{U_{i}\}_{i \in I} \vee \{U_{i}\}_{i \in J}
&=\{U_{i}\}_{i \in I \amalg J}, \\
\{U_{i}\}_{i} \wedge \{V_{j}\}_{j}
&=\{U_{i} \times_{R} V_{j}\}_{ij}
\end{eqnarray*}
This does not depend on the representation,
and $\Omega_{1}(R)$ becomes a distributive lattice:
$0$ is an empty set, and $1$ is $\{R \to R\}$.
\item Let $A \to R$ be a morphism in $\scr{C}$.
Then, we have a morphism
$\Omega_{1}(R) \to \Omega_{1}(A)$
defined by $\{U_{i}\}_{i} \mapsto \{A \times_{R} U_{i}\}_{i}$.
This gives a contravariant functor $\Omega_{1}:\scr{C}^{\op}
\to \cat{DLat}$.
We will denote the functor
$\func{pt}\func{comp}\Omega_{c}(-):\scr{C} \to \cat{Coh}$
by $\Spec^{0}$.
\end{enumerate}
\end{Def}
\begin{proof}
Here, we will only give a brief proof of (2) and (4).
In the sequel, we omit the subscript of $\times_{R}$
and simply write $\times$.
\begin{enumerate}
\setcounter{enumi}{1}
\item We first show the reflexivity:
$\{U_{i}\}_{i} \prec \{U_{i}\}_{i}$.
For any $j$, $U_{j} \times U_{j}$ is isomorphic to $U_{j}$,
since $U_{j} \to R$ is monic.
Hence, $\{U_{j} \to U_{j}\} \in \mathcal{O}_{U_{j}}$.
From the condition (b) of Definition \ref{def:groth:top}, we have that 
$\{U_{i} \times U_{j} \to U_{j}\}_{j} \in \mathcal{O}_{U_{j}}$.

Next, we show the transitivity:
suppose $\{U_{i}\}_{i} \prec \{V_{j}\}_{j}$ and
$\{V_{j}\}_{j} \prec \{W_{k}\}_{k}$.
Then, $\{W_{k} \times V_{j} \to V_{j}\}_{k} \in \mathcal{O}_{V_{j}}$ for any $j$,
and (c) shows that $\{W_{k} \times V_{j} \times U_{i} \to V_{j} \times U_{i}\}_{k}
\in \mathcal{O}_{U_{i} \times V_{j}}$ for any $i,j$.
Combining with the fact that 
$\{U_{i} \times V_{j} \to U_{i}\}_{j} \in \mathcal{O}_{U_{i}}$
for any $i$, (d) tells that $\{U_{i}\}_{i} \prec \{W_{k}\}_{k}$.
\setcounter{enumi}{3}
\item We will only prove for the join.
Suppose $\{U_{i}\}_{i} \prec \{\tilde{U}_{k}\}_{k}$
and $\{V_{j}\}_{j} \prec \{\tilde{V}_{l}\}_{l}$.
Then, $\{V_{j} \times \tilde{V}_{l} \to V_{j}\}_{l} \in \mathcal{O}_{V_{j}}$
tells that 
\[
\{U_{i} \times \tilde{U}_{k} \times V_{j} \times \tilde{V}_{l}
\to U_{i} \times \tilde{U}_{k} \times V_{j}\}_{l} \in 
\mathcal{O}_{U_{i} \times \tilde{U}_{k} \times V_{j}}.
\]
Also, $\{\tilde{U}_{k} \times U_{i} \to U_{i}\}_{k} \in \mathcal{O}_{U_{i}}$
implies
\[
\{U_{i} \times \tilde{U}_{k} \times V_{j} \to U_{i} \times V_{j}\}_{k}
\in \mathcal{O}_{U_{i} \times V_{j}}.
\]
Combining these two, we obtain 
$\{U_{i} \times V_{j}\}_{ij} \prec \{\tilde{U}_{k} \times \tilde{V}_{l}\}_{kl}$.
\end{enumerate}
\end{proof}

\begin{Def}
\label{def:weak:sch}
\begin{enumerate}
\item On a coherent space $X$,
we have a canonical $\cat{DLat}$-valued sheaf 
$\tau_{X}:\Omega_{c}(X) \to \cat{DLat}^{\op}$
defined by $U \mapsto \Omega_{c}(U)$.
A morphism $f:X \to Y$ of coherent spaces
induces a morphism $f^{-1}:\tau_{Y} \to f_{*}\tau_{X}$
of $\cat{DLat}$-valued sheaves on $X$,
defined by
\[
\Omega_{c}(U) =\tau_{Y}(U) \to \tau_{X}(f^{-1}U)=\Omega_{c}(f^{-1}U)
\quad (V \mapsto f^{-1}V).
\]
\item On a $\scr{C}^{\op}$-valued coherent space
$(X,\scr{O}_{X})$,
we have a canonical $\cat{DLat}$-valued sheaf
$\sigma_{X}:\Omega_{c}(X) \to \cat{DLat}^{\op}$
defined by the sheafification of $U \mapsto \Omega_{1}(\scr{O}_{X}(U))$.
\item A \textit{weak $\scr{C}$-scheme} is a triple
$X=(X,\scr{O}_{X},\beta_{X})$, where $(X,\scr{O}_{X})$ is 
a $\scr{C}^{\op}$-coherent space,
and $\beta_{X}:\sigma_{X} \to \tau_{X}$ is a morphism
of $\cat{DLat}^{\op}$-valued sheaves on $X$, which satisfies the
following condition:

for any inclusion $V \subset U$ of quasi-compact open
subsets of $X$, the restriction functor $\scr{O}_{X}(V) \to \scr{O}_{X}(U)$
factors through the $\scr{E}$-morphism $j:A \to \scr{O}_{X}(U)$
whenever $\beta_{X}(U)(\{A \to \scr{O}_{X}(U)\}) \geq V$.

We will refer to $\beta_{X}$ as the ``support morphism" on $X$:
we will explain the reason of this name later in \ref{exam:Asch:Nagata}.

Also, we will denote the underlying coherent space by $|X|$ to avoid confusion. 
\item A morphism $f:X \to Y$ of weak $\scr{C}$-schemes
is a pair $(f,f^{\#})$, where $f:|X| \to |Y|$ is a quasi-compact morphism,
and $f^{\#}:f_{*}\scr{O}_{X} \to \scr{O}_{Y}$ is a morphism
of $\scr{C}^{\op}$-valued sheaves on $Y$ which
makes the following diagram commutative:
\[
\xymatrix{
\sigma_{Y} \ar[r]^{f^{\#}} \ar[d]_{\beta_{Y}} & f_{*}\sigma_{X} \ar[d]^{f_{*}\beta_{X}} \\
\tau_{Y} \ar[r]_{f^{-1}} & f_{*}\tau_{X}
},
\]
where the top arrow is defined by
\[
\sigma_{Y}(U)=\Omega_{1}(\scr{O}_{Y}(U)) 
\stackrel{\Omega_{1}f^{\#}}{\to} \Omega_{1}(\scr{O}_{X}(f^{-1}U))
=\sigma_{X}(f^{-1}U).
\]
\item We denote by $\cat{w.$\scr{C}$-Sch}$
the category of weak $\scr{C}$-schemes and their morphisms.
\end{enumerate}
\end{Def}
One advantage of this new definition is that,
we can obtain the spectrum functor as the adjoint
of the global section functor, as below.
\begin{Def}
The global section functor $\Gamma:\cat{w.$\scr{C}$-Sch} \to \scr{C}$
admits a right adjoint, namely the spectrum functor $\Spec$.
Moreover, $\Spec$ is fully faithful.
\end{Def}
\begin{proof}
For an object $R \in \scr{C}$, we define a weak $\scr{C}$-scheme
$X=(|X|,\scr{O}_{X},\beta_{X})$ as follows:
\begin{enumerate}[(a)]
\item $|X|=\Spec^{0} R$,
\item $\scr{O}_{X}$ is a $\scr{C}^{\op}$-valued sheaf on $X$,
in other words, a finite continuous functor
$\Omega_{1}(R) \to \scr{C}$. This is defined by
\[
U=\{U_{i}\}_{i} \mapsto \coker(\amalg_{i,j}(U_{i} \times_{R} U_{j})
\stackrel{p_{1},p_{2}}{\rightrightarrows} \amalg_{i}U_{i}).
\]
This does not depend on the representation of the open set $U$.
For an inclusion $U=\{U_{i}\}_{i} \prec \{V_{j}\}_{j}=V$ of quasi-compact
open subsets of $|X|$,
we have a commutative diagram
\[
\xymatrix{
\amalg_{k,l}(V_{k} \times V_{l} \times U_{i}) \ar@<.5ex>[r] \ar@<-.5ex>[r] \ar[d] &
\amalg_{k}(V_{k} \times U_{i}) \ar[r] \ar[d] & U_{i} \ar@{.>}[d] \\
\amalg_{k,l}(V_{k} \times V_{l}) \ar@<.5ex>[r] \ar@<-.5ex>[r] 
& \amalg_{k} V_{k} \ar[r] & \scr{O}_{X}(V)
},
\]
which patches up to give $\scr{O}_{X}(U) \to \scr{O}_{X}(V)$.

This is the \textit{structure sheaf of $X$}.
\item The support morphism $\beta_{X}:\sigma_{X} \to \tau_{X}$
is defined canonically: any quasi-compact open subset $U$
of $|X|$ is represented by a family of $\scr{E}$-morphisms
$\{U_{i}\to R\}$. Set $A=\scr{O}_{X}(U)$.
Then, for any $\scr{E}$-morphism $V \to A$, the composition
\[
U_{i} \times_{A} V \to U_{i} \to U
\]
is an $\scr{E}$-morphism.
The family of $\scr{E}$-morphisms $\{U_{i} \times_{A} V \to U\}$
defines a quasi-compact open subset $\beta_{X}(U)(V)$
of $U$, and is independent
of the choice of the representation of $U$.
\end{enumerate}
Let $A \to B$ be a morphism in $\scr{C}$.
We can define a morphism
\[
X=\Spec A \to \Spec B =Y
\]
as follows: we already have the morphism between the underlying spaces
$f:|X| \to |Y|$. The morphism between
the structure sheaves $f^{\#}:f_{*}\scr{O}_{X} \to \scr{O}_{Y}$
is given as follows: for a quasi-compact open set $U=\{U_{i}\}$ of $Y$,
we have a commutative diagram
\[
\xymatrix{
\amalg_{i,j}(A \times_{B} U_{i} \times_{B} U_{j}) 
\ar[d] \ar@<.5ex>[r] \ar@<-.5ex>[r] &
\amalg_{i}(A \times_{B}U_{i}) \ar[d] \ar[r] &
\scr{O}_{X}(f^{-1}U) \ar@{.>}[d] \\
 \amalg_{i,j} (U_{i} \times_{B} U_{j}) 
 \ar@<.5ex>[r] \ar@<-.5ex>[r] &
\amalg_{i}U_{i}\ar[r] &
\scr{O}_{Y}(U)  \\
}
\]
which gives a morphism $\scr{O}_{X}(f^{-1}U) \to \scr{O}_{Y}(U)$.
We can easily check that $(f,f^{\#})$ is indeed
a morphism of weak $\scr{C}$-schemes.

Hence, we have a functor $\Spec: \scr{C} \to \cat{w. $\scr{C}$-sch}$.
The unit functor $\epsilon_{X}:X \to \Spec \Gamma(X,\scr{O}_{X})=\tilde{X}$
is given as follows: the support morphism $\beta_{X}(X)$
gives a lattice homomorphism 
$\Omega_{1}(\Gamma(X,\scr{O}_{X})) \to \Omega_{c}(X)$, which induces the morphism 
$\epsilon:|X| \to \Spec^{0}\Gamma(X,\scr{O}_{X})$
between the underlying spaces.
The morphism between the structure sheaves
$\epsilon^{\#}_{X}:\epsilon_{*}\scr{O}_{X} \to \scr{O}_{\tilde{X}}$
is given as follows:
for each $\scr{E}$-morphism $U_{i} \to \Gamma(X,\scr{O}_{X})$,
there is a morphism $\scr{O}_{X}(\epsilon^{-1}U_{i}) \to \scr{O}_{\tilde{X}}(U_{i})$.
These patch up to give a morphism $\epsilon_{*}\scr{O}_{X} \to \scr{O}_{\tilde{X}}$.

The counit functor is the canonical morphism
$\eta_{R}:R \to \Gamma(\Spec^{0}R, \scr{O}_{\Spec R})$.
Since $\eta$ is a natural isomorphism,
$\Spec$ is fully faithful.
\end{proof}

\begin{Def}
\begin{enumerate}
\item A weak $\scr{C}$-scheme which is isomorphic to a spectrum
of some object in $\scr{C}$ is called \textit{affine}.
\item
A \textit{$\scr{C}$-scheme} is a weak $\scr{C}$-scheme
which is locally affine.
Let $\cat{$\scr{C}$-Sch}$ be the full subcategory of $\cat{w.$\scr{C}$-Sch}$
consisting of $\scr{C}$-schemes.
\item We have a pullback in $\cat{$\scr{C}$-Sch}$ constructed
as follows: let $X,Y$ be a $\scr{C}$-schemes over a $\scr{C}$-scheme $S$.
\begin{enumerate}[(i)]
\item Case $X,Y,S$ are all affine, say $X=\Spec A$,
$Y=\Spec B$, $S=\Spec R$ respectively.
Then $X \times_{S}Y$ is defined as $\Spec(A \times_{R} B)$.
\item Case $S$ is affine:
let $X=\cup_{i}X_{i}$ and $Y=\cup_{j}Y_{j}$ be open affine coverings
of $X$ and $Y$. Then, $\{X_{i} \times_{S} Y_{j}\}_{ij}$ patches up to
define $X \times_{S} Y$.
\item Case $S$ is arbitrary:
Let $S=\cup_{l} S_{l}$ be an open affine covering of $S$,
and set $X_{l}=\pi_{X}^{-1}(S_{l})$, $Y_{l}=\pi_{Y}^{-1}(S_{l})$
where $\pi_{X}:X \to S$ and $\pi_{Y}:Y \to S$ are structure maps.
Then, $\{X_{l} \times_{S_{l}} Y_{l}\}_{l}$ patches up to give
$X \times_{S} Y$.
\end{enumerate}
These definitions does not depend on the choice of affine open covers
of $X,Y$ and $S$.

\item The lluf subcategory $\scr{E}^{\prime}$ of $\cat{$\scr{C}$-Sch}$
are open immersions, and $\mathcal{O}^{\prime}$
is the family of finite quasi-compact open covers.
Then, $(\cat{$\scr{C}$-Sch}, \scr{E}^{\prime},\mathcal{O}^{\prime})$
 becomes a schematic coherent site,
and the spectrum functor $\Spec:\scr{C} \to \cat{$\scr{C}$-Sch}$
becomes a morphism of coherent sites.
\item Let $F:\scr{C} \to \scr{D}$ be a morphism
of coherent sites.
Then, we can extend $F$ to a functor $\cat{$\scr{C}$-Sch} \to \cat{$\scr{D}$-Sch}$
as follows:
\begin{enumerate}[(i)]
\item If $X=\Spec A$ is an affine $\scr{C}$-scheme,
then $F(X)=\Spec F(A)$.
\item For a general $\scr{C}$-scheme $X$,
let $X=\cup_{i} U_{i}$ be an affine open covering of $X$.
Then $F(X)$ is defined by the patching of $F(U_{i})$.
This definition does not depend on the choice of the affine covering $\{U_{i}\}$.
\item Let $f:X \to Y$ be a morphism of $\scr{C}$-schemes,
and $Y=\cup_{i}Y_{i}$ be an affine open covering of $Y$.
Then, $X_{i}=f^{-1}(Y_{i}) \to Y_{i}$ is determined
by a morphism $\Gamma(X_{i},\scr{O}_{X}) \to B_{i}=\Gamma(Y_{i},\scr{O}_{Y})$.
This induces a morphism $\Gamma(F(X_{i}),\scr{O}_{F(X)})
\to F(B_{i})=\Gamma(F(Y_{i}),\scr{O}_{F(Y)})$,
and these patch up to give $F(f):F(X) \to F(Y)$.
\item In other words, we have a functor $\func{Sch}:
\cat{CohSite} \to \cat{CohSch}$,
sending $\scr{C}$ to $\cat{$\scr{C}$-Sch}$.
\end{enumerate}
\end{enumerate}
\end{Def}

\begin{Rmk}
\begin{enumerate}
\item
This definition of schemes is seemingly different
from the conventional one:
we usually use the notion of locally ringed spaces
and local homomorphisms, and in fact,
the above definition is almost equivalent to this
conventional one.
However, the usual definition of schemes bothers us since
it uses the limit process to look at the stalks.
Note that we cannot take arbitrary limits in the categories
of usual schemes.
\item There is another fancier way of defining 
$\scr{C}$-schemes.
$\scr{C}$ becomes a site by the fixed
Grothendieck topology,
and $\scr{C}$-scheme is a set-valued sheaf on $\scr{C}$
which locally isomorphic to
\[
\scr{C}(-,R):A \mapsto \scr{C}(A,R)
\]
for some $R \in \scr{C}$.
When we consider groupoid-valued sheaves,
then we can immediately define $\scr{C}$-stacks
(\cite{LMB}).
\item The notion of weak $\scr{C}$-schemes coincides
with $\scr{A}$-schemes introduced in \cite{Takagi2},
when $\scr{C}$ is the opposite category of rings,
monoids, etc. We will summarize the advantage of considering
$\scr{A}$-schemes and weak $\scr{C}$-schemes later on.
\end{enumerate}
\end{Rmk}

\begin{Prop}
\label{prop:sch:idem}
If $\scr{C}$ is schematic, then 
the functor $\Spec:\scr{C} \to \cat{$\scr{C}$-Sch}$
is an equivalence of categories.
\end{Prop}
\begin{proof}
It suffices to show that any $\scr{C}$-scheme $X$ is affine.
Let $X=\cup_{i}^{n}U_{i}$ be an affine open covering of $X$.
It suffices to show for $n=2$,
since the rest follows from the induction.
Let $U_{1} \cap U_{2}=\cup_{j}^{m} V_{j}$ be an affine
open covering of $U_{1} \cap U_{2}$.
Again, it suffices to show for $m=1$.

Therefore, the proof is reduced to the following Lemma:
\end{proof}
\begin{Lem}
Suppose the following diagram is a bicartesian square
in $\scr{E}$ of $\scr{C}$:
\[
\xymatrix{
A \ar[r]^{u} \ar[d]_{v} & B \ar[d] \\
C \ar[r] & D
}
\]
Then, applying the spectrum functor to this square gives
a bicartesian square in $\cat{$\scr{C}$-Sch}$.
\end{Lem}
\begin{proof}
Suppose we are given a commutative diagram
\[
\xymatrix{
\Spec A \ar[r] \ar[d] & \Spec B \ar[d]^{f} \\
\Spec C \ar[r]_{g} & X 
}
\]
in $\cat{$\scr{C}$-Sch}$.
We can define the unique morphism $h:\Spec D \to X$ by an explicit construction:
\[
h:\Omega_{c}(X) \to \Omega_{1}(D)=\Omega_{1}(B) \times_{\Omega_{1}(A)} \Omega_{1}(C)
\]
is given by $U \mapsto (f^{-1}U,g^{-1}U)$.
The morphism $h_{*}\scr{O}_{\Spec D} \to \scr{O}_{X}$ is given by
\[
\scr{O}_{\Spec D}(h^{-1}U)
=\scr{O}_{B}(f^{-1}U) \amalg_{\scr{O}_{A}((fu)^{-1}U)} \scr{O}_{C}(g^{-1}U)
\to \scr{O}_{X}(U).
\]
\end{proof}

We have come to the stage of proving 
Theorem \ref{thm:univ:coh:sch}.
\begin{proof}[Proof of Theorem \ref{thm:univ:coh:sch}]
We only have to show that the functor $\func{Sch}$
is the left adjoint of $U:\cat{CohSch} \to \cat{CohSite}$.
The unit $\epsilon:\Id \Rightarrow U\func{Sch}$ is the spectrum functor
$\epsilon_{\scr{C}}=\Spec:\scr{C} \to \cat{$\scr{C}$-Sch}$.
The counit $\eta:\func{Sch}U \Rightarrow \Id$ is defined by the inverse
of $\Spec$, since $\Spec:\scr{D} \to \cat{$\scr{D}$-Sch}$
is an equivalence for schematic coherent site $\scr{D}$
by Proposition \ref{prop:sch:idem}.
\end{proof}

\section{Comparison theorems}
\subsection{The correspondence between the topology and ideals}

In the previous sections, we obtained
the distributive lattice from open immersions.
However, in the usual definition we approach from
the ideal description.
In this section, we will explain how to
obtain the Grothendieck topology from ideals.

\begin{Def}
Let $V=\langle \Omega, E\rangle$ be an algebraic type and
$\scr{C}$ a category of $V$-algebras.
$V$ has \textit{self enhancing property}, if:
\begin{enumerate}[(a)]
\item $\scr{C}(A,B)$ has a natural structure
of $V$-algebras for any two objects $A,B \in \scr{C}$,
\item the composition is bilinear, namely 
the natural morphisms
$f^{*}:\scr{C}(B,X) \to \scr{C}(A,X)$
and $f_{*}:\scr{C}(X,A) \to \scr{C}(X,B)$
induced from any homomorphism $f:A \to B$
of $V$-algebras are also homomorphisms.
\end{enumerate}
\end{Def}

\begin{Rmk}
$V$ has self enhancing property if 
any two operators $\phi,\psi$ are \textit{commutative}, namely
if $\phi$ is $m$-ary and $\psi$ is $n$-ary,
then 
\begin{multline*}
\psi(\phi(x_{11},\cdots,x_{m1}),\cdots \phi(x_{1n},\cdots,x_{mn})) \\
=\phi(\psi(x_{11},\cdots,x_{1n}),\cdots,\psi(x_{m1},\cdots,x_{mn})).
\end{multline*}
Then, for any $A,B \in \scr{C}$,
the action of the operators on $\scr{C}(A,B)$ can be defined,
by the action on the value of homomorphisms:
\[
\phi(f_{1},\cdots,f_{n})(x)=\phi(f_{1}(x),\cdots,f_{n}(x))
\]
for any homomorphisms $f_{i} \in \scr{C}(A,B)$
and any $n$-ary operator $\phi$.
\end{Rmk}

\begin{Prop}
Let $V$ be an algebraic type with self enhancing property,
$\scr{C}$ the category of $V$-algebras,
and $M$ a $V$-algebra.
Then, the functor $\scr{C}(M,-):\scr{C} \to\scr{C}$ admits a left adjoint
$M \otimes (-)$.
\end{Prop}
\begin{proof}
For any $N \in \scr{C}$,
let $S$ be the set of quotient $V$-algebras
of $\func{free}(M \times N)$, where $\func{free}:\cat{Set} \to \scr{C}$
is the free generator.
Suppose we are given a homomorphism $\varphi:N \to \Hom(M,X)$.
Then, we have a homomorphism $\tilde{\varphi}:\func{free}(M \times N) \to X$
induced by the map $M \times N \to X$, $(x,y) \mapsto \varphi(y)(x)$.
Let $A \in S$ be the image of $\tilde{\varphi}$.
Then we have a commutative diagram
\[
\xymatrix{
N \ar[r]^{\varphi} \ar[d] & \scr{C}(M,X) \\
\scr{C}(M,A) \ar[ru]_{\tilde{\varphi}_{*}}
}
\]
Since $S$ is a small set, we can apply Freyd's
adjoint functor theorem \cite{CWM}
to claim the existence of the left adjoint of $\Hom(M,-)$.
\end{proof}
As in the case of abelian categories,
$(\scr{C}, \otimes)$ becomes a closed
symmetric monoidal category:
the unit $\mathbf{1}$ with respect to $\otimes$ is $\func{free}(*)$,
and the composition $\scr{C}(A,B) \times \scr{C}(B,C) \to \scr{C}(A,C)$
naturally extends to $\scr{C}(A,B) \otimes \scr{C}(B,C) \to \scr{C}(A,C)$.

In the sequel, we fix an algebraic type $V$ with self enhancing property,
and $\scr{C}$ the category of $V$-algebras.
\begin{Def}
\begin{enumerate}
\item An object $R \in \scr{C}$ is a \textit{monoid object}
(resp. \textit{commutative monoid object})
if there are morphisms $R \otimes R \to R$
and $\mathbf{1} \to R$ satisfying the 
monoid (resp. commutative monoid) axioms.
We denote by $\cat{Mnd/$\scr{C}$}$ (resp. $\cat{CMnd/$\scr{C}$}$)
the category of monoid (resp. commutative monoid)
 objects in $\scr{C}$.
 Note that this notion is different from what we have
 defined in the preliminary (\S 1);
 therefore, we use different notations.
\item Let $R \in \scr{C}$ be a monoid object.
An object $M \in \scr{C}$ is an \textit{$R$-module} if
there is an $R$-action $R \otimes M \to M$ with the unital law.
When $R$ is commutative,
the category $\cat{$R$-mod}$ of $R$-modules shares
similar properties with $\scr{C}$, namely
it has self enhancing property, so that 
the symmetric monoidal structure $\otimes_{R}$ is well defined
and is closed.
Also, there is a free generator 
$\func{free}:\cat{Set} \to \cat{$R$-mod}$.
\item An \textit{ideal} of a commutative monoid object
$R$ is an $R$-submodule of $R$.
\item For two finitely generated
ideals $\mathfrak{a},\mathfrak{b}$ of a commutative monoid object
$R$, the multiplication
$\mathfrak{a}\mathfrak{b}$ is defined by the image
of the morphism $\mathfrak{a} \otimes_{R} \mathfrak{b} \to R$.
$\mathfrak{a}+\mathfrak{b}$ is defined by the image of the morphism
$\mathfrak{a} \amalg \mathfrak{b} \to R$.
These two operations make the set $I(R)$ of 
finitely generated ideals of $R$
into a semiring.
\end{enumerate}
\end{Def}
We have a set-theoretic description of monoid
objects: $R$ is a monoid object in $\scr{C}$
if $R$ has a multiplicative monoid structure which satisfies
the distribution laws:
\[
\begin{split}
a \cdot \phi(b_{1},\cdots, b_{n}) &=\phi(ab_{1},\cdots,ab_{n}), \\
\phi(a_{1},\cdots,a_{n})\cdot b & =\phi(a_{1}b,\cdots,a_{n}b) 
\end{split}
\]
for any $n$-ary operator $\phi$.

Note that for some $V$, the emptyset may become an ideal
of a commutative monoid object:
for example when $V$ is null.

\begin{Exam}
The followings are algebraic types $V$ with self enhancing property:
$W$ is the algebraic type of commutative monoid object
of $V$-algebras
\begin{enumerate}
\item $V$ is the null algebraic type, namely $V$-algebras are sets.
In this case, $\otimes$ is the cartesian product,
and $W$ is the type of commutative monoids,
in the usual sense.
\item $V$ is the type of pointed sets.
Then $\otimes$ is the smash product,
and $W$ is the type of commutative monoids
with an absorbing element.
\item $V$ is the type of abelian groups.
Then $\otimes$ is the usual tensor product, and
$W$ is the type of commutative rings.
\item $V$ is the type of commutative 
idempotent monoids with absorbing elements.
$W$ is the type of semirings,
except that we do not assume that $1$ is absorbing
with respect to $+$. 
\end{enumerate}
\end{Exam}

\begin{Def}
\label{def:congruence}
Let $R$ be a $V$-algebra.
\begin{enumerate}
\item
A \textit{congruence} $\equiv$ of $R$ is an equivalence
relation induced by a morphism $f:R \to A$ of $V$-algebras:
\[
a \equiv b \Leftrightarrow f(a)=f(b).
\]
\item We have a left adjoint $\func{rad}:\cat{Semiring} \to \cat{DLat}$
of the underlying functor $\cat{DLat} \to \cat{Semiring}$.
The unit morphism $\epsilon:I(R) \to \func{rad}(I(R))$ induces a congruence $E$ on $I(R)$.
The \textit{radical ideal} of $R$ is an ideal $\mathfrak{j}$
of $R$ such that
\[
a \in \mathfrak{j} \Leftrightarrow (a,0) \in E.
\]
\end{enumerate}
\end{Def}

\begin{Lem}
\label{lem:equiv:dlat}
Let $R$ be a semiring, and $a,b \in R$.
Then, $a=b$ in $\func{rad}(R)$ if and only if there is an integer $n$
such that $a^{n} \leq b$ and $b^{n} \leq a$.
\end{Lem}
\begin{proof}
Let $\equiv$ be an equivalence relation
on $R$ defined by
\[
a \equiv b \Leftrightarrow a^{n} \leq b, b^{n} \leq a \quad (n \gg 0).
\]
Set $A=R/\equiv$. Then,
the semiring structure of $R$ descends to $A$,
and $A$ becomes a distributive lattice.
Conversely, it is obvious that the canonical map $R \to \func{rad}(R)$
factors through $A$, hence $A \simeq \func{rad}(R)$.
\end{proof}

Let $R \in \cat{CMnd/$\scr{C}$}$ be a commutative monoid object,
$M$ an $R$-module
and $S \subset R$ a multiplicative subset of $R$.
The \textit{localization} $S^{-1}M$ is defined by the universal property
of making the actions of the elements of $S$ invertible.
We can give an explicit construction $S^{-1}M$
in the usual way: $S^{-1}M$ is the quotient set of $M \times S$
divided by the equivalence
\[
(x,s) \equiv (y,t) \Leftrightarrow \exists u \in S, [utx=usy].
\]
As usual, $S^{-1}R$ also becomes a commutative monoid
object, and $S^{-1}M=S^{-1}R \otimes_{R} M$.
Any $R$-multilinear operator $\phi$ on $M$ can be extended
to a $S^{-1}R$-multilinear operator on $S^{-1}M$:
\[
\phi\left(\Frac{x_{1}}{s_{1}},\cdots, \Frac{x_{n}}{s_{n}}\right)
=\Frac{1}{\prod_{i=1}^{n}s_{i}}
\phi\left(\left(\prod_{j \neq 1}s_{j}\right)x_{1}, \cdots,
\left(\prod_{j \neq n}s_{j}\right)x_{n}\right).
\]
If $S$ is generated by a single element $f$, then $S^{-1}R$
is denoted by $R_{f}$.

\begin{Lem}
\label{lem:ideal:derive:op}
Let $R$ be a commutative monoid object in $\scr{C}$,
and $\mathfrak{a}=(f_{1},\cdots,f_{n})$
be a finitely generated ideal of $R$.
Then, the element of $\mathfrak{a}$ can be written in a
form $\phi(a_{1}f_{1},\cdots,a_{r}f_{n_{r}})$,
where $\phi$ is a derived operator in $\scr{C}$
and $a_{i} \in R$.
In particular, for any element $x \in I$,
there exists an $R$-equivariant map
$\psi:R^{\times n} \to R$ such that
$\psi(f_{1},\cdots,f_{n})=x$.
Here, $R^{\times n}$ is a product of $n$-copies of $R$,
regarded as an $R$-module.
\end{Lem}
This can be easily
proven by using the induction on the length of derived operators
and the distribution law.

\begin{Thm}
Let $R$ be a commutative monoid object in $\scr{C}$,
and suppose $\mathfrak{a}=(f_{i})_{i}$ be a finitely generated ideal 
of $R$ such that $\epsilon(\mathfrak{a})=1$ in $\func{rad}(I(R))$.
Then $\{R_{f_{i}}\}_{i}$ covers $R$, namely
\[
R \to \prod_{i}R_{f_{i}} \stackrel{p_{1},p_{2}}{\rightrightarrows} 
\prod_{i,j}R_{f_{i}f_{j}}
\]
is exact.
\end{Thm}
\begin{proof}
Let $u=(x_{i}/f_{i}^{n_{i}})_{i} \in \prod_{i} R_{f_{i}}$ be an element
of the equalizer, namely $p_{1}(u)=p_{2}(u)$.
This is equivalent to saying that
$x_{i}/f_{i}^{n_{i}}=x_{j}/f_{j}^{n_{j}}$ in $R_{f_{i}f_{j}}$.
We may assume that $f_{j}x_{i}=f_{i}x_{j}$ by
replacing $f_{i}$'s by sufficiently large powers;
this is a standard method, and can be seen in Hartshorne 
(\cite{Harts}, II 2.2).
Since $(f_{i})_{i}^{N}=1$ in $I(R)$
for sufficiently large $N$
by Lemma \ref{lem:equiv:dlat}, we have $\psi(f_{1},\cdots,f_{n})=1$
for some $R$-equivariant map $\psi:R^{\times n} \to R$
by Lemma \ref{lem:ideal:derive:op}.
Set $x=\psi(x_{1},\cdots,x_{n}) \in R$. Then,
\[
f_{i}x=\psi(f_{i}x_{j})_{j}=\psi(f_{j}x_{i})_{j}=x_{i}\psi(f_{j})_{j}=x_{i}.
\]
This implies that $x=x_{i}/f_{i}$ in $R_{f_{i}}$ so that
$u$ is in the image of $R \to \prod_{i}R_{f_{i}}$.
\end{proof}

\begin{Lem}
\label{lem:sieve:cond:local}
Let $R$ be a commutative monoid object in $\scr{C}$,
and $f_{i},g_{j}$ elements of $R$.
If $(f_{i})_{i}=1$ in $I(R)$ and
$(f_{i}g_{j})_{j}=1$ in $I(R_{f_{i}})$ for any $i$,
then $(g_{j})_{j}=1$ in $I(R)$.
\end{Lem}
\begin{proof}
For every $i$, we have a $R$-equivariant map
$\psi_{i}:R_{f_{i}}^{n} \to R_{f_{i}}$ such that
$\psi_{i}(f_{i}g_{j})_{j}=1$ in $R_{f_{i}}$.
Multiplying by some power of $f_{i}$ if necessary,
we obtain a derived operator $\tilde{\psi}_{i}:R^{n} \to R$
such that $\tilde{\psi}_{i}(a_{m}g_{j_{m}})_{m}=f_{i}^{N}$ for some $N$
and some $a_{m} \in R$.
Since $(f_{i})_{i}=1$ in $I(R)$,
we also have $(f_{i}^{N})_{i}=1$.
Hence, there is a derived operator $\phi:R^{m} \to R$
such that $\phi(b_{l}f_{i_{l}}^{N})_{i}=1$ for some $b_{l} \in R$.
Composing $\tilde{\psi}_{i}$'s and $\phi$ gives
an $R$-equivariant map $\gamma:R^{n} \to R$
such that $\gamma(g_{j})_{j}=1$, hence the result.
\end{proof}

\begin{Cor}
\label{cor:zar:top}
Let $\scr{E}$ be the lluf subcategory of $\scr{C}^{\op}$
consisting of localizations of finite type.
For each $R \in \scr{C}$,
let $\mathcal{O}_{R}$ be the family of sets $\{R_{f_{i}} \to R\}_{i}$
of $\scr{E}$-morphisms such that the ideal $(f_{i})_{i}$ is unital.
Then, $(\scr{C}^{\op},\scr{E},\mathcal{O})$
satisfies the axioms of coherent sites,
and the induced spectrum functor $\Spec^{0}:\scr{C}^{\op} \to \cat{Coh}$
is the desired one:
\[
\Spec^{0}R=\func{pt}\func{comp}\func{rad}(I(R)).
\]
\end{Cor}
In the sequel, we will refer to this
Grothendieck topology as the \textit{Zariski topology}.

\begin{proof}
First, we show that $(\scr{C}^{\op},\scr{E},\mathcal{O})$
is indeed a coherent site.
It is easy to see that $R_{f} \to R$ is indeed monic
and flat in $\scr{C}^{\op}$.
The descent condition is satisfied by definition,
hence we only need to verify that $\mathcal{O}$
is a Grothendieck topology.
The only non-trivial statement is condition (d)
in Definition \ref{def:groth:top},
but this is proved in Lemma \ref{lem:sieve:cond:local}.

Next, we show that the topology coincides
with that obtained from ideals.
Let $R$ be a commutative monoid object of $\scr{C}$,
and $U=\{R_{f_{i}} \to R\}_{i}$ a quasi-compact open set
of $\Spec^{0}(R)$.
We define the corresponding radical ideal $\varphi(U) \in \func{rad}(I(R))$
to be the ideal generated by $(f_{i})_{i}$.
This does not depend on the representation of $U$,
hence $\varphi:\Omega_{1}(R) \to \func{rad}(I(R))$ is well defined.
Conversely, let $\mathfrak{a}=(f_{i})_{i}\in \func{rad}(I(R))$ be a radical ideal.
The quasi-compact open set $\psi(\mathfrak{a})$
is defined by $\{R_{f_{i}} \to R\}$.
This again does not depend on the choice of the generators
of $\mathfrak{a}$, hence we have a bijection
between $\Omega_{1}(R)$ and $\func{rad}(I(R))$.
Also, it is straightforward to see that this is a lattice
isomorphism, and that it is functorial with respect to $R$.
\end{proof}

\begin{Rmk}
For some algebraic category $\scr{C}$ with its Zariski topology,
$\scr{C}$ is already schematic, namely it is vacuous to consider
$\scr{C}$-schemes.
For the followings, $\scr{C}$ is the opposite category of $W$-algebras
(with the Zariski topology; see Corollary \ref{cor:zar:top}).
The category $\scr{C}$ becomes schematic, if any localization is surjective.
The following cases are such:
\begin{enumerate}
\item $W$ is the type of distributive lattices,
\item $W$ is the type of idempotent semirings,
\item $W$ is the algebraic type of von Neumann regular commutative rings:
a commutative ring $R$ is \textit{von Neumann regular},
if any element $x\in R$ has a \textit{weak inverse} $y \in R$,
namely the unique element $y \in R$
which satisfies $x^{2}y=x$ and $xy^{2}=y$.
The Krull dimension of a von Neumann regular ring is $0$,
hence any localization becomes surjective.
This is related to Serre's cohomological criterion of affineness: 
a scheme $X$ is affine
if and only if $H^{i}(X, \scr{F})=0$ for 
any quasi-coherent sheaf $\scr{F}$ and $i>0$.

Hence, any von Neumann regular scheme is affine,
since its Krull dimension is $0$.
\end{enumerate}
\end{Rmk}

\subsection{Comparison with classical schemes}

In this section, we discuss local objects
of a complete coherent site
(a coherent site $(\scr{C},\scr{E},\mathcal{O})$ is \textit{complete}
if $\scr{C}$ is complete),
and show that the notion of $\scr{C}$-scheme coincides with
\begin{enumerate}
\item the notion of coherent schemes, when $\scr{C}$
is the opposite category of rings, and
\item the notion of monoid schemes,
in the sense of To\"{e}n-Vaqui\'{e}, or Deitmar,
when $\scr{C}$ is the opposite category of commutative monoids.
\end{enumerate}

\begin{Def}
Let $\scr{C}$ be a complete coherent site.
\begin{enumerate}
\item An object $A \in \scr{C}$ is \textit{local},
if any covering of $A$ is principal, namely:
$U_{i}=A$ for some $i$ if $\{U_{i} \to A\}_{i}$ is a covering of $A$.
This is equivalent to saying that $\Spec^{0}(A)$
has a unique closed point.
\item A morphism $f:A \to B$
between two local objects $A, B \in \scr{C}$
is \textit{local}, if the induced morphism $f:\Spec^{0} A \to \Spec^{0}B$
sends the unique closed point of $\Spec^{0}A$ to that of $\Spec^{0}B$.
In the language of distributive lattices, this is equivalent to
\[
f^{-1}(U)=1 \Rightarrow U=1
\] 
for any quasi-compact open $U \subset B$.
\item A $\scr{C}$-coherent space $(X, \scr{O}_{X})$
is \textit{local}, if the stalk $\scr{O}_{X,x}=\varprojlim_{x \in U}\scr{O}_{X}(U)$
is local for any $x \in X$.
\item A morphism $f:X \to Y$ of $\scr{C}^{\op}$-spaces
is \textit{local}, if the induced morphism $f:\scr{O}_{X,x} \to \scr{O}_{Y,f(x)}$
on the stalks is local.
\item We denote by $\cat{$\scr{C}$-LCS}$ the category
of locally $\scr{C}^{\op}$-spaces and local morphisms.
\end{enumerate}
\end{Def}

The following proposition shows that the category of $\scr{C}$-schemes
coincides with the notion of coherent schemes, when
$\scr{C}^{\op}$ is the opposite category of commutative rings: 
\begin{Thm}
\label{thm:sch:local:space}
We have a natural fully faithful functor
$\cat{w. $\scr{C}$-Sch} \to \cat{$\scr{C}$-LCS}$.
\end{Thm}
\begin{proof}
Let $X=(X,\scr{O}_{X},\beta_{X})$ be a weak $\scr{C}$-scheme.
First, we will show that $X$ is a local $\scr{C}^{\op}$-space.
Fix a point $x \in X$, and let $Y=\Spec^{0} \scr{O}_{X,x}=\cup_{i} U_{i}$
be a quasi-compact open cover of $Y$.
We have a natural isomorphism 
$Y=\varprojlim_{x \in U} \Omega_{c}(\Spec \scr{O}_{X}(U))$
since $\Spec$ is a right adjoint,
and $\beta_{X}$ induces a morphism 
\[
\iota: X_{x}= \varprojlim_{x \in U}U \to Y.
\]
The left-hand side is a local object in $\cat{DLat}$.
Since $\{U_{i}\}_{i}$ covers $Y$, one of $\iota^{-1}U_{i}$ must coincide with $X_{x}$.
Since $\Omega_{1}(X_{x})$ is obtained by localizations of open neighborhoods
of $x$, $\iota^{-1}U_{i}=V$ for some quasi-compact open neighborhood $V$ of $x$.
This shows that $U_{i}=V$ in $\Omega_{1}(\scr{O}_{X}(V))$,
and hence also in $\Omega_{1}(\scr{O}_{X,x})$.
Therefore, $X$ is a local $\scr{C}^{\op}$-space.

Suppose we have a morphism $f:X \to Y$ of weak $\scr{C}$-schemes.
We need to show that the induced morphism $\scr{O}_{X,x} \to \scr{O}_{Y,y}$
is local, for every $x \in X$ and $y=f(x)$.
We have a commutative diagram
\[
\xymatrix{
\Omega_{1}(\scr{O}_{Y,y}) \ar[r]^{f^{-1}} \ar[d]_{\beta_{Y}} &
 \Omega_{1}(\scr{O}_{X,x}) \ar[d]^{\beta_{X}} \\
\Omega_{c}(Y_{y}) \ar[r] & \Omega_{c}(X_{x})
}
\]
of distributive lattices.
Suppose $U \in \Omega_{1}(\scr{O}_{Y,y})$ maps to $1 \in \Omega_{1}(\scr{O}_{X,x})$.
Then, $\beta_{X}(f^{-1}U)=1$: this implies that $x \in \beta_{X}(f^{-1}U)$,
and hence $y \in \beta_{Y}(U)$.
This shows that $U=1$, which means that $\scr{O}_{X,x} \to \scr{O}_{Y,y}$
is indeed a local morphism.
As a consequence, we have a functor 
$\iota:\cat{w.$\scr{C}$-Sch} \to \cat{$\scr{C}$-LCS}$.

It remains to prove that $\iota$ is fully faithful.
Let $X,Y$ be two weak $\scr{C}$-schemes,
and $f:X \to Y$ a morphism of locally $\scr{C}^{\op}$-spaces.
For any inclusion of two quasi-compact open subsets $V \subset U$ of $Y$,
it suffices to show that $f^{-1}(\beta_{Y}(U)(V))=\beta_{X}(f^{\#}(U)(V))$.
Let $x$ be any point in $\beta_{X}(f^{\#}(U)(V))$,
and set $y=f(x)$. We have a sequence of local morphisms
\[
\Omega_{1}(\scr{O}_{Y,y}) \to \Omega_{1}(\scr{O}_{X,x}) \to \Omega_{c}(X_{x}),
\]
which shows that $V=1$ in $\Omega_{1}(\scr{O}_{Y,y})$.
We have another sequence of local morphisms
\[
\Omega_{1}(\scr{O}_{Y,y}) \to \Omega_{c}(Y_{y}) \to \Omega_{c}(X_{x}),
\]
induced by $f^{-1}\circ \beta_{Y}$, and this sends $1$ to $1$.
This is equivalent to saying that $x \in f^{-1}(\beta_{Y}(U)(V))$.
Therefore, $f^{-1}(\beta_{Y}(U)(V)) \supset \beta_{X}(f^{\#}(U)(V))$.
The converse can be proved similarly.
\end{proof}

Hence, we see that the notion of $\scr{C}$-scheme coincides
with that of conventional schemes, when we set $\scr{C}$
and Grothendieck topology properly.
\begin{Exam}
\begin{enumerate}
\item
Let $V$ be the algebraic type of abelian groups.
Then, the commutative monoid objects in the category
of $V$-algebras are exactly commutative rings.
For a ring $R$, $I(R)$ is the set of \textit{finitely generated} ideals
of $R$.
Note that $\func{comp}\func{rad}(R(I))$ is the set of radical ideals,
and $\func{pt}\func{comp}\func{rad}(R(I))$ is the usual 
Zariski spectrum of $R$.

Let $\scr{C}=\cat{CRing}^{\op}$ be the opposite category
of commutative unital rings, and $\scr{E}$ the lluf subcategory
of localizations of finite type.
Then we see that $\scr{C}$-schemes are exactly 
coherent schemes.

\item
When $V$ is the null algebraic type,
the the commutative monoid objects in the category
of $V$-algebras are exactly commutative monoids.
Then, we obtain the coherent monoid schemes, in the sense
of Deitmar, of equivalently, that of To\"{e}n-Vaqui\'{e} \cite{TV}.
\end{enumerate}
\end{Exam}
\begin{Rmk}
The category of coherent schemes include
all noetherian schemes and affine schemes,
which are practically all the schemes we treat.
Moreover, the morphisms between
noetherian (or affine) schemes are all quasi-compact.
This tells that the category of coherent schemes
is sufficiently large to work on.
Also, Remark \ref{rmk:sch:cat:incomp} implies
that we had better work within this category.
\end{Rmk}

\subsection{Comparison with $\scr{A}$-schemes}

In this section,
we compare weak $\scr{C}$-schemes with $\scr{A}$-schemes
introduced in \cite{Takagi1}.
We recall the definition of $\scr{A}$-schemes,
in the most general type.
\begin{Def}
A quadruple $\scr{A}=(W, \alpha_{1},\alpha_{2},\gamma)$
is a \textit{schematizable algebraic type}, if
\begin{enumerate}[(a)]
\item $W$ is an algebraic type with commutative
multiplicative monoid structure.
We denote by $\cat{$W $-alg}$ the
category of $W$-algebras and their homomorphisms.
\item $\alpha_{1}$ is a functor $\cat{$W $-alg} \to \cat{DLat}$.
\item $\alpha_{2}$ is a natural transformation
$\Id_{\cat{$W $-alg}} \Rightarrow \alpha_{1}$
such that $\alpha_{2,R}:R \to \alpha_{1}R$
is multiplication-preserving.
\item For each $W$-algebra $R$ and a multiplicative system
$S$, $\gamma$ is a natural isomorphism
$\alpha_{1}(S^{-1}R) \to (\alpha_{2,R}(S))^{-1}R$.
\end{enumerate}
\end{Def}

When $\scr{F}$ is a $\cat{$W$-alg}$-valued
sheaf on a coherent space $X$,
then the functor $\alpha_{1}$ induces a $\cat{DLat}$-valued
sheaf $\alpha_{1}\scr{F}$ on $X$, defined by the sheafification
of $U \mapsto \alpha_{1}\scr{F}(U)$.

\begin{Def}
Let $\scr{A}=(W,\alpha_{1},\alpha_{2},\gamma)$
be a schematizable algebraic type.
An \textit{$\scr{A}$-scheme} is a triple $(X,\scr{O}_{X},\beta_{X})$
such that 
\begin{enumerate}
\item $X$ is a coherent space,
$\scr{O}_{X}$ is a $\cat{$W $-alg}$-valued sheaf on $X$,
$\beta_{X}:\alpha_{1}\scr{O}_{X} \to \tau_{X}$
is a morphism of $\cat{DLat}$-valued sheaves on $X$
which we refer to as the ``support morphism", and
\item the restriction maps reflect localizations:
let $V \subset U$ be an inclusion of quasi-compact open subsets of $X$,
and we denote by $\scr{O}_{X}(U)_{V}$ the localization
of $\scr{O}_{X}(U)$ by the multiplicative system
\[
\{f \in \scr{O}_{X}(U) \mid \beta_{X}\alpha_{2}(f) \geq V\}.
\]
\end{enumerate}
A \textit{morphism} $f:X \to Y$ is a pair $(f,f^{\#})$,
where 
\begin{enumerate}[(a)]
\item $f:|X| \to |Y|$ is a morphism of coherent spaces,
\item $f^{\#}:\scr{O}_{Y} \to f_{*}\scr{O}_{X}$
is a morphism of $\cat{$W $-alg}$-valued sheaves
on $Y$, such that the following diagram commutes:
\[
\xymatrix{
\alpha_{1}\scr{O}_{Y} \ar[r]^{\alpha_{1}f^{\#}}
\ar[d]_{\beta_{Y}} &
f_{*}\alpha_{1}\scr{O}_{X} \ar[d]^{f_{*}\beta_{X}} \\
\tau_{Y} \ar[r]_{f^{-1}} & f_{*}\tau_{X}
}
\]
\end{enumerate}
\end{Def}

Here, we list up the properties of $\scr{A}$-schemes:
let $\cat{$\scr{A}$-Sch}$ be the category
of $\scr{A}$-schemes and their morphisms.
Then:
\begin{Thm}[\cite{Takagi2}]
\begin{enumerate}
\item $\cat{$\scr{A}$-Sch}$
is a full subcategory of locally $\cat{$W $-alg}$-spaces.
\item The category of $\scr{A}$-schemes
is small complete and small co-complete.
\item We can consider the image $\scr{A}$-scheme
for any morphism of $\scr{A}$-schemes.
\end{enumerate}
\end{Thm}

Practically, we only consider the following case:
\begin{Exam}
\label{exam:Asch:alg:type}
$V$ is a self enhancing algebraic type,
and $W$ is the algebraic type of commutative monoid
objects in $V$-algebras.
For each $W$-algebra $R$,
$\alpha_{1}(R)$ is the set of finitely generated ideals of $R$,
divided by the congruence defined by
\[
\mathfrak{a} \equiv \mathfrak{b}
\Leftrightarrow \mathfrak{a}^{n} \leq \mathfrak{b}, \ 
\mathfrak{b}^{n} \leq \mathfrak{a} \quad (n \gg 0)
\]
(cf. Lemma \ref{lem:equiv:dlat}).
$\alpha_{2}:R \to \alpha_{1}(R)$ sends $a \in R$
to the principal ideal generated by $a$.
This preserves multiplication.
For each multiplicative system $S$ of $R$,
we have a natural isomorphism $\gamma:
\alpha_{1}(S^{-1}R) \simeq \alpha_{2}(S)^{-1}\alpha_{1}(R)$.
\end{Exam}
The set $\alpha_{1}(R)$ can be regarded
as the distributive lattive corresponding to $\Spec R$,
and we identify a finitely generated ideal
$\mathfrak{a} \in \alpha_{1}(R)$ with the quasi-compact
open subset which is the complement of the support of $\mathfrak{a}$.
Note that $\alpha_{1}\scr{O}_{X}$ corresponds
to $\sigma_{X}$ appeared in the definition of weak
$\scr{C}$-schemes (cf. Definition \ref{def:weak:sch}).

Then, it is straightforward to see that
\begin{Thm}
Let $\scr{A}=(W,\alpha_{1},\alpha_{2},\gamma)$
be as above.
Then, the category of $\scr{A}$-schemes is equivalent
to weak $\scr{C}$-schemes defined in section \S 5.
\end{Thm}

\begin{Exam}
\label{exam:Asch:Nagata}
\begin{enumerate}
\item Suppose $\scr{A}$ is the schematizable
algebraic type induced from rings:
namely, $V$ is the algebraic type of abelian groups
in Example \ref{exam:Asch:alg:type}.
Then for each ring $R$,
$\alpha_{1}(R)$ is the idempotent
semiring of finitely generated ideals of $R$
modulo the congruence generated by $\mathfrak{a}^{2}=\mathfrak{a}$.
For each element $\mathfrak{a} \in \alpha_{1}(R)$,
$\beta_{X}$ simply gives the complement of
the support of $\mathfrak{a}$ which is a quasi-compact 
open subset of $\Spec R$.
This is why we call $\beta_{X}$ `the support morphism'.

Actually, the data $\beta_{X}$ implies that 
$(X,\scr{O}_{X})$ is a locally ringed space
by Theorem \ref{thm:sch:local:space}.

\item Let $\scr{A}$ be as in (1).
Let $U$ be a non-empty open subset
of a coherent scheme $X$.
The Zariski-Riemann space $\ZR(U,X)$
is the limit of $U$-admissible blowups,
hence an object of $\scr{A}$-schemes,
but in general not a scheme.
However, this object happens to be useful
to prove pure algebro-geometric theorems,
such as Nagata embeddings (\cite{Takagi2}, \cite{Takagi3}).
\item There is a notion of Zariski-Riemann
spaces for monoids.
This is used to prove the existence of 
equivariant compactification of toric varieties
in \cite{EI}.
\end{enumerate}
\end{Exam}

\section{The property of the Zariski topology}

Let $V$ be a self enhancing algebraic type,
$W$ the type of commutative monoid objects in $V$-algebras,
and $R$ a $W$-algebra.
So far, we have introduced the Zariski topology 
on the category of $W$-algebras, which is defined
by the ideals of $R$, without the reason why.
In fact, there is no natural reason for this choice,
if we only aim to construct a natural $\cat{$W$-alg}$-valued
space. In addition, 
we have at least two natural topologies,
different from the Zariski topology.

Therefore, we need to characterize
Zariski topology by a universal property,
if we want to claim it as a natural one.
This is accomplished only by looking at modules over
the spectrum.

In the sequel,
$V$ is a self enhancing algebraic type, and
$W$ is the type of commutative monoid objects in $V$-algebras.

\subsection{Modules}

The importance of looking at
$R$-modules instead of $R$ itself is already
well known, and typically it appears in the Morita theory:
\begin{Thm}[Morita]
Let $V$ be a self enhancing algebraic type,
and $R$ be a monoid object in $\cat{$V$-alg}$.
Then, we have a natural isomorphism
\[
Z(R) \simeq \End(\Id_{\cat{mod-$R$}})
\]
of monoid objects in $\cat{$V$-alg}$,
where $Z(R)$ is the center of $R$,
and $\cat{mod-$R$}$ is the category
of right $R$-modules.
In particular, $R$ can be recovered from the category of 
$R$-modules when $R$ is commutative.
\end{Thm}
\begin{proof}
In the sequel, any $R$-module is a right $R$-module.

Let $F:\Id \to \Id$ be a natural endomorphism
of $\Id=\Id_{\cat{mod-$R$}}$.
Then $F_{R}:R \to R$ is given by
a left multiplication by an element $z \in R$,
since $R$ is a free $R$-module generated by one element.
Since $F$ is natural, $F_{R}$
must commute with any left multiplication
$a \times (-):R \to R$ by an element $a \in R$.
This shows that $z$ is in the center of $R$.
The universal property of coproducts tells
that $F_{P}:P \to P$ is also a left multiplication by $z$,
for any free $R$-module $P$.

Let $M$ be an arbitrary $R$-module.
Then we have a surjective homomorphism $P \to M$
from a free $P$-module.
The commutative diagram
\[
\xymatrix{
P \ar[r] \ar[d]_{z \times (-)} & M \ar[d]^{F_{M}} \\
P \ar[r] & M
}
\]
tells that $F_{M}$ is also a left multiplication by $z$.
\end{proof}

Suppose there is a surjective homomorphism
$R \to M$ of $R$-modules when $R$ is a ring.
Then this map is essentially determined by the kernel,
which is an ideal of $R$.
In contrast, the surjective map cannot be
parametrized by the ideal of $R$,
if we consider algebraic types other than that of rings;
we need congruences (see Definition \ref{def:congruence}).

Therefore, at first sight we may think that
we should replace ideals by congruences.
However, we run into trouble when
we do this for general cases.
\begin{Exam}
The typical example if when $V$ is the
null algebraic type, namely when an $R$-module
is a set with an action of a monoid $R$.
Then, the set of 
congruences of $R$ may not have
a maximal element,
since the set of all congruences is not finitely generated.
\end{Exam}
This means that, if we try to construct a
complete distributive lattice from the congruence directly,
then the space may not have enough points. 

One possible solution is to pass through
other algebraic objects, such as rings,
so that we can obtain a distributive lattice
from congruences.
This idea is realized in \cite{Deitmar2}.

However, it happens that, looking at ideals is
sufficient to detect modules under a moderate assumption.

\begin{Prop}
\label{prop:congruence:max}
Suppose $V$ has a constant operator $0$
(then for any $W$-algebra $R$,
$0 \in R$ becomes the absorbing element
with respect to the multiplication).
 
Let $R$ be any non-zero $W$-algebra,
and $\mathfrak{a}$ be a congruence of $R$ as an $R$-module.
\begin{enumerate}
\item $\mathfrak{a}=R\times R$ if and only if $(1,0) \in \mathfrak{a}$.
\item Any non-unital congruence $\mathfrak{a}$
is a subcongruence of a maximal congruence.
\end{enumerate}
\end{Prop}
\begin{proof}
\begin{enumerate}
\item Let $M=R/\mathfrak{a}$ be the quotient $R$-module.
Then for any $x \in M$,
\[
x=1 \cdot x=0 \cdot x=0.
\]
\item This follows from the fact that
the set $\mathcal{S}$ of non-unital congruences of $R$
including $\mathfrak{a}$ is an inductively ordered set,
since the unit congruence is principally generated by $(1,0)$.
\end{enumerate}
\end{proof}
Actually, if the algebraic type does not
contain a constant operator, then
we can add it, so that the situation is reduced to
the above case.

\subsection{Quasi-coherent modules}

In the sequel, $\scr{C}$ is the opposite category of $W$-algebras,
$\scr{E}$ is the subcategory of $\scr{C}$
consisting of localizations of finite type;
this restriction on $\scr{E}$ is reasonable, since
\begin{enumerate}
\item it is the simplest and the most important
case that occurs, and
\item
in general, to describe
flat epimorphisms (the condition which
we require for $\scr{E}$) of $W$-algebras
is not at all trivial,
even for the classical case of rings;
see \cite{Lazard}, \cite{BBR}.
\end{enumerate}

\begin{Def}
We fix a coherent Grothendieck topology $\mathcal{O}$
on $\scr{C}$.
We have the corresponding functor 
$\Spec^{\mathcal{O}}:\scr{C} \to \cat{$\scr{C}$-Sch}$.
\begin{enumerate}
\item Let $R$ be a $W$-algebra and $X=\Spec^{\mathcal{O}}R$.
Let $M$ be an $R$-module.
The $\scr{O}_{X}$-module $\Shf^{\mathcal{O}}(M)$
is defined by the sheafification of $A \mapsto A \otimes_{R}M$,
where $A \to R$ is a morphism in $\scr{E}$.
\item Let $X$ be a $\scr{C}$-scheme.
An $\scr{O}_{X}$-module is \textit{quasi-coherent},
if it is isomorphic to $\Shf^{\mathcal{O}}(M)$
on each affine open subset $\Spec^{\mathcal{O}}R$ of $X$,
where $M$ is an $R$-module.
\end{enumerate}
\end{Def}

When $\mathcal{O}$ is the Zariski topology,
we denote by $\Spec^{\Zar}R$ (resp. $\Shf^{\Zar}(M)$)
for $\Spec^{\mathcal{O}}R$ (resp. $\Shf^{\mathcal{O}}(M)$).

When there is an inclusion
$\mathcal{O}_{1} \to \mathcal{O}_{2}$
between two topologies,
this induces an immersion 
$\Spec^{\mathcal{O}_{2}}R \to \Spec^{\mathcal{O}_{1}}R$
on the underlying space.

\begin{Def}
\begin{enumerate}
\item For any $W$-algebra $R$,
let $\mathcal{O}^{\Tot}_{R}$ be the set
of all finite sets $\{R[f_{i}^{-1}] \to R\}_{i}$
of localizations which satisfies the descent datum.
The family $\mathcal{O}^{\Tot}=\{\mathcal{O}^{\Tot}_{R}\}_{R}$
gives a Grothendieck topology.
\item For any $W$-algebra $R$,
let $\mathcal{O}^{\min}_{R}$ be the set
of all finite sets $\{R[f_{i}^{-1}] \to R\}_{i}$
of localizations such that either one of the $f_{i}$'s is a unit.
The family $\mathcal{O}^{\min}=\{\mathcal{O}^{\min}_{R}\}_{R}$
gives a Grothendieck topology.
\end{enumerate}
\end{Def}
The topology $\mathcal{O}^{\Tot}$ (resp. $\mathcal{O}^{\min}$)
gives the coarsest (resp. finest) topology:
when we denote by $\Spec^{\Tot}$ (resp. $\Spec^{\min}$)
the corresponding spectrum functors, then
we have immersions
\[
\Spec^{\Tot}R \to \Spec^{\Zar}R
\to \Spec^{\min}R.
\]
These immersions are not isomorphisms in general.
For example, when $R$ is a noetherian commutative ring,
then the points of $\Spec^{\Tot}R$ correspond
to the points of $\Spec R$ of height less than $2$
via Krull's Hauptidealsatz (\cite{Matsumura}, Theorem 13.5).

Therefore,
we need to consider modules to characterize
the Zariski topology.
\begin{Thm}
\label{thm:ideal:top:mod}
Suppose $V$ has a constant operator $0$.
Then, the Zariski topology is the coarsest topology $\mathcal{O}$
which satisfies the following condition:

(*) For any $W$-algebra $R$ and any $R$-module $M$,
$\Shf^{\mathcal{O}}(M)=0$ if and only if $M=0$.
\end{Thm}
\begin{proof}
First, we will show that $\Shf^{\Zar}(M)=0$ implies $M=0$.
Let $x \in M$ be a non-zero element.
We have an isomorphism $R/\mathfrak{a} \to Rx$
of $R$-modules, where $\mathfrak{a}$ is a congruence on $R$
as an $R$-module.
Since $(1,0) \notin \mathfrak{a}$,
there is a maximal congruence $\mathfrak{m}$ containing $\mathfrak{a}$
by Proposition \ref{prop:congruence:max}.
Let $\mathfrak{n}$ be the subset of $R$ consisting of elements
$a$ such that $(a,0) \in \mathfrak{m}$.
We see that $\mathfrak{n}$ is a prime ideal,
and the stalk $\Shf^{\Zar}(M)_{\mathfrak{n}}$ is non-zero.

Next, we will show that the Zariski topology is indeed
the coarsest topology satisfying (*).
Let $\mathcal{O}$ be a topology satisfying (*).
Suppose $\{R[f_{i}^{-1}] \to R\}_{i} \in \mathcal{O}_{R}$
is not in $\mathcal{O}^{\Zar}_{R}$,
namely $(f_{i})_{i}$ does not generate the unit ideal in $R$.
Let $J$ be the congruence on $R$
generated by $\{(0,f_{i})\}_{i}$.
Then $M=R/J$ is not a zero $R$-module.
However, $\Shf^{\mathcal{O}}(M)|_{U_{i}}=0$,
where $U_{i}=\Spec^{\mathcal{O}}R[f_{i}^{-1}]$
is an open subset of $X=\Spec^{\mathcal{O}}R$ for each $i$.
By definition, $U_{i}$'s cover $X$.
Therefore, $\Shf^{\mathcal{O}}(M)=0$, a contradiction.
\end{proof}


\textsc{S. Takagi: Faculty of Science,
Osaka City University, Osaka, 558-8585, Japan}

\textit{E-mail address}: \texttt{takagi@sci.osaka-cu.ac.jp}
\end{document}